\documentclass[oneside,english]{amsart}
\usepackage[T1]{fontenc}
\usepackage[latin9]{inputenc}
\usepackage{amsthm}
\usepackage{amstext}
\usepackage{amssymb}
\usepackage{esint}

\makeatletter
\numberwithin{equation}{section}
\numberwithin{figure}{section}
\theoremstyle{plain}
\newtheorem{thm}{\protect\theoremname}[section]

\theoremstyle{plain}
\theoremstyle{definition}

\theoremstyle{plain}
\newtheorem{lem}[thm]{\protect\lemmaname}

\theoremstyle{plain}

\theoremstyle{plain}
\makeatother

\usepackage{babel}
\usepackage{color}
\providecommand{\definitionname}{Definition}
\providecommand{\lemmaname}{Lemma}
\providecommand{\theoremname}{Theorem}
\providecommand{\corollaryname}{Corollary}
\providecommand{\remarkname}{Remark}
\providecommand{\propositionname}{Proposition}

\DeclareMathOperator{\loc}{loc}

\DeclareMathOperator{\ess}{ess}
\DeclareMathOperator{\cp}{cap}

\DeclareMathOperator{\ACL}{ACL}
\DeclareMathOperator{\diam}{diam}

\begin{document}

\title[On geometric characterizations of mappings]
{On geometric characterizations of mappings generate composition operators on Sobolev spaces}

\author{Alexander Ukhlov}
\begin{abstract}
In this work we consider refined geometric characterizations of mappings generate composition operators on Sobolev spaces. The detailed proofs in the cases $n-1<q<n$ and $n>q$ are given.
\end{abstract}
\maketitle

\section{Introduction}
In this work we consider refined geometric characterizations \cite{GGR95,VU98} of mappings generate composition operators on Sobolev spaces. Recall that quasiconformal mappings allow the geometric description in the terms of geometric dilatations \cite{Ge60} and are closely connected with composition operators on Sobolev spaces \cite{VG75}. The bounded composition operators on Sobolev spaces arise in the Sobolev embedding theory \cite{GG94,GS82} and have applications in the weighted Sobolev spaces theory \cite{GU09} and in the spectral theory of elliptic operators \cite{GU17}. The theory of multipliers in connections with composition operators was considered in \cite{MS86}. In \cite{U93,VU98} were given various characteristics of homeomorphisms $\varphi: \Omega\to\widetilde{\Omega}$, where $\Omega, \widetilde{\Omega}$ are domains in $\mathbb R^n$, which generate by the composition rule $\varphi^{\ast}(f)=f\circ\varphi$ the bounded embedding operators on Sobolev spaces:
\begin{equation}
\label{pq}
\varphi^{\ast}:L^1_p(\widetilde{\Omega})\to L^1_q(\Omega),\,\,1<q\leq p<\infty.
\end{equation}

The mappings generate bounded composition operators \eqref{pq} are called as weak $(p,q)$-quasiconformal mappings \cite{GGR95,VU98} because in the case $p=q=n$ we have usual quasiconformal mappings \cite{VG75}. In \cite{U93,VU98} it was proved that the homeomorphism $\varphi:\Omega\to\widetilde{\Omega}$ is the weak $(p,q)$-quasiconformal mapping,  if and only if $\varphi\in W^1_{1,\loc}(\Omega)$, has finite distortion and
$$
K_{p,q}^{\frac{pq}{p-q}}(\varphi;\Omega)=\int\limits_{\Omega}\left(\frac{|D\varphi(x)|^p}{|J(x,\varphi)|}\right)^{\frac{q}{p-q}}~dx<\infty,
\,\,1< q<p< \infty.
$$
and
$$
K_{p,p}^{p}(\varphi;\Omega)=\ess\sup\limits_{\Omega}\frac{|D\varphi(x)|^p}{|J(x,\varphi)|}<\infty,\,\,1< q=p< \infty.
$$
In the case $1< q=p< \infty$ such mappings are called as a weak $p$-quasiconformal mappings \cite{GGR95}.

The capacitory characterizations of weak $(p,q)$-quasiconformal mappings were given in \cite{U93,VU98}. It was proved that the homeomorphism $\varphi:\Omega\to\widetilde{\Omega}$ is the weak $(p,q)$-quasiconformal mapping,  if and only if   
the inequalities
$$
\cp_{p}^{1/p}(\varphi^{-1}(F_0),\varphi^{-1}(F_1);\Omega) \leq K_{p,p}(\varphi;\Omega)\cp_{p}^{1/p}(F_0,F_1;\widetilde{\Omega})
$$
and
$$
\cp_{q}^{1/q}(\varphi^{-1}(F_0),\varphi^{-1}(F_1);\Omega)
\leq\widetilde{\Phi}(\widetilde{\Omega}\setminus(F_0\cup F_1))^{\frac{p-q}{pq}}
\cp_{p}^{1/p}(F_0,F_1;\widetilde{\Omega})
$$
where $\widetilde{\Phi}$ is a bounded monotone countable-additive set function defined on open subsets of $\widetilde{\Omega}$,
hold for every condenser $(F_0,F_1)\subset \widetilde{\Omega}$.

The aim of the present work is to give the refined characterizations of weak $(p,q)$-quasiconformal mappings in the terms of the geometric dilatation
$$
H_{p}^{\lambda}(x,r)=\frac{L_{\varphi}^{p}(x,r)r^{n-p}}{|\varphi(B(x,\lambda r))|},\,\, \lambda\geq 1,
$$
where $L_{\varphi}(x,r)=\max\limits_{|x-y|=r}|\varphi(x)-\varphi(y)|$, with detailed proofs. 

The first time geometric characterizations of weak $p$-quasiconformal mappings, $p\ne n$, were introduced in \cite{GGR95}, but without detailed proofs. The geometric characterizations of weak $(p,q)$-quasiconformal mappings on Carnot groups were considered in \cite{VU98}, without the special description ($\lambda=1$) in the case case $n<q<p<\infty$. The geometric characterizations in the Euclidean case $\mathbb R^n$ were considered in the manuscript \cite{U94}. 

Remark that geometric characterizations of weak $p$-quasiconformal mappings can be defined on metric measure spaces and so can be used in the geometric analysis on metric measure spaces.

The author is grateful to Vladimir Gol'dshtein for useful discussions and valuable remarks on geometric properties of generalized quasiconformal mappings. 

\section{Composition operators on Sobolev spaces}

\subsection{Sobolev spaces}

Let us recall the basic notions of the Sobolev spaces.
Let $\Omega$ be an open subset of $\mathbb R^n$. The Sobolev space $W^1_p(\Omega)$, $1\leq p\leq\infty$, is defined \cite{M}
as a Banach space of locally integrable weakly differentiable functions
$f:\Omega\to\mathbb{R}$ equipped with the following norm: 
\[
\|f\mid W^1_p(\Omega)\|=\| f\mid L_p(\Omega)\|+\|\nabla f\mid L_p(\Omega)\|,
\]
where $\nabla f$ is the weak gradient of the function $f$, i.~e. $ \nabla f = (\frac{\partial f}{\partial x_1},...,\frac{\partial f}{\partial x_n})$. The Sobolev space $W^{1}_{p,\loc}(\Omega)$ is defined as a space of functions $f\in W^{1}_{p}(U)$ for every open and bounded set $U\subset  \Omega$ such that $\overline{U}  \subset \Omega$.

The homogeneous seminormed Sobolev space $L^1_p(\Omega)$, $1\leq p\leq\infty$, is defined as a space
of locally integrable weakly differentiable functions $f:\Omega\to\mathbb{R}$ equipped
with the following seminorm: 
\[
\|f\mid L^1_p(\Omega)\|=\|\nabla f\mid L_p(\Omega)\|.
\]

In the Sobolev spaces theory, a crucial role is played by capacity as an outer measure associated with Sobolev spaces \cite{M}. In accordance to this approach, elements of Sobolev spaces $W^1_p(\Omega)$ are equivalence classes up to a set of $p$-capacity zero \cite{MH72}.

Recall that a function $f:\Omega\to\mathbb R$ belongs to the class $\ACL(\Omega)$ if it is absolutely continuous on almost all straight lines which are parallel to any coordinate axis. Note that $f$ belongs to the Sobolev space $W^1_{1,\loc}(\Omega)$ if and only if $f$ is locally integrable and it can be changed by a standard procedure (see, e.g. \cite{M} ) on a set of measure
zero (changed by its Lebesgue values at any point where the Lebesgue values exist) so that a modified function belongs to $\ACL(\Omega)$, and its partial derivatives $\frac{\partial f}{\partial x_i}$, $i=1,...,n$, existing a.e., are locally integrable in $\Omega$.

The mapping $\varphi:\Omega\to\mathbb{R}^{n}$ belongs to the Sobolev space $W^1_{p,\loc}(\Omega)$, if its coordinate functions belong to $W^1_{p,\loc}(\Omega)$. In this case, the formal Jacobi matrix $D\varphi(x)$ and its determinant (Jacobian) $J(x,\varphi)$
are well defined at almost all points $x\in\Omega$. The norm $|D\varphi(x)|$ is the operator norm of $D\varphi(x)$. Recall that a mapping $\varphi:\Omega\to\mathbb{R}^{n}$ belongs to $W^1_{p,\loc}(\Omega)$, is a mapping of finite distortion if $D\varphi(x)=0$ for almost all $x$ from $Z=\{x\in\Omega: J(x,\varphi)=0\}$ \cite{VGR}. 
Recall the notion of of the variational $p$-capacity associated with Sobolev spaces \cite{GResh}. The condenser in the domain $\Omega\subset \mathbb R^n$ is the pair $(F_0,F_1)$ of connected closed relatively to $\Omega$ sets $F_0,F_1\subset \Omega$. A continuous function $u\in L_p^1(\Omega)$ is called an admissible function for the condenser $(F_0,F_1)$,
if the set $F_i\cap \Omega$ is contained in some connected component of the set $\operatorname{Int}\{x\vert u(x)=i\}$,\ $i=0,1$. We call $p$-capacity of the condenser $(F_0,F_1)$ relatively to domain $\Omega$
the value
$$
{{\cp}}_p(F_0,F_1;\Omega)=\inf\|u\vert L_p^1(\Omega)\|^p,
$$
where the greatest lower bond is taken over all admissible for the condenser $(F_0,F_1)\subset\Omega$ functions. If the condenser have no admissible functions we put the capacity is equal to infinity.

\subsection{Composition operators}

Let $\Omega$ and $\widetilde{\Omega}$ be domains in the Euclidean space $\mathbb R^n$. Then a homeomorphism $\varphi:\Omega\to\widetilde{\Omega}$ generates a bounded composition
operator 
\[
\varphi^{\ast}:L^1_p(\widetilde{\Omega})\to L^1_q(\Omega),\,\,\,1\leq q\leq p\leq\infty,
\]
by the composition rule $\varphi^{\ast}(f)=f\circ\varphi$, if for
any function $f\in L^1_p(\widetilde{\Omega})$, the composition $\varphi^{\ast}(f)\in L^1_q(\Omega)$
is defined quasi-everywhere in $\Omega$ and there exists a constant $K_{p,q}(\varphi;\Omega)<\infty$ such that 
\[
\|\varphi^{\ast}(f)\mid L^1_q(\Omega)\|\leq K_{p,q}(\varphi;\Omega)\|f\mid L^1_p(\widetilde{\Omega})\|.
\]

Recall that the $p$-dilatation \cite{Ger69} of a Sobolev mapping $\varphi: \Omega\to \widetilde{\Omega}$ at the point $x\in\Omega$ is defined as
$$
K_p(x)=\inf \{k(x): |D\varphi(x)|\leq k(x) |J(x,\varphi)|^{\frac{1}{p}}\}.
$$

\begin{thm}
\label{CompTh} Let $\varphi:\Omega\to\widetilde{\Omega}$ be a homeomorphism
between two domains $\Omega$ and $\widetilde{\Omega}$. Then $\varphi$ generates a bounded composition
operator 
\[
\varphi^{\ast}:L^1_p(\widetilde{\Omega})\to L^1_{q}(\Omega),\,\,\,1> q\leq p\leq\infty,
\]
 if and only if $\varphi\in W^1_{q,\loc}(\Omega)$
and 
\[
K_{p,q}(\varphi;\Omega) := \|K_p \mid L_{\kappa}(\Omega)\|<\infty, \,\,1/q-1/p=1/{\kappa}\,\,(\kappa=\infty, \text{ if } p=q).
\]
The norm of the operator $\varphi^\ast$ is estimated as $\|\varphi^\ast\| \leq K_{p,q}(\varphi;\Omega)$.
\end{thm}

This theorem in the case $p=q=n$ was given in the work \cite{VG75}. The general case $1\leq q\leq p<\infty$ was proved in \cite{U93}, where the weak change of variables formula \cite{H93} was used (see, also the case $n<q=p<\infty$ in \cite{V88}).

\section{Geometric characterizations of mappings}

Let us recall the following covering lemma \cite{Va}:
\begin{lem}
\label{lem:VaCovering}
Let $F$ be a compact subset of $\mathbb R^1$. Then for every $\varepsilon>0$ there exists a number
$\delta>0$ such that for any $r\in (0,\delta)$ there exists a finite covering of $F$
by open intervals $\gamma_1,\gamma_2,\dots,\gamma_N$
such that

\noindent
{$1)$}
$|\gamma_i|=2r$, for all $1\leq i\leq N$;

\noindent
{$2)$}
centers of intervals $\gamma_i$ belong to $F$;

\noindent
{$3)$}
any point of the set $F$ belongs no more than two intervals $\gamma_i$;

\noindent
{$4)$} $Nr\leq |F|+\varepsilon$.
\end{lem}

Let us recall the definition of the Hausdorff measure~$H^{\alpha}$. Let $A\subset \mathbb R^n$ be arbitrary set. Then for an arbitrary $r>0$ we
consider a countable covering $\{U_i\}$ of the set $A$ such that $\diam(U_i)<r$ for all $i$.
We put
$$
H_{r}^{\alpha}(A)=\inf\{\sum\limits_i(\diam U_i)^{\alpha}\},
$$
where the greatest lower bond is taken over all such coverings. The function $H_{r}^{\alpha}$ does not increase by $r$.
The Hausdorff measure is defined
$$
H^{\alpha}(A)=\lim\limits_{r\to 0}H_{r}^{\alpha}(A).
$$

\subsection{Weak $p$-quasiconformal mappings}

Let $\varphi :\Omega\to \widetilde{\Omega}$ be a homeomorphism. Follow \cite{GGR95} we introduce the geometric $p$-dilatation
$$
H_{p}^{\lambda}(x,r)=\frac{L_{\varphi}^{p}(x,r)r^{n-p}}{|\varphi(B(x,\lambda r))|},\,\, \lambda\geq 1,
$$
where $L_{\varphi}(x,r)=\max\limits_{|x-y|=r}|\varphi(x)-\varphi(y)|$.

\begin{thm} \label{thm:ACLPPconf} Let $\varphi :\Omega\to \widetilde{\Omega}$ be a homeomorphism which satisfies
$$
\limsup\limits_{r\to 0} H_{p}^{\lambda}(x,r)\leq H_p^{\lambda}<\infty, \,\,\text{for all}\,\,x\in\Omega,\,\,1<p<\infty.
$$
Then the homeomorphism $\varphi$ belongs to $\ACL(\Omega)$. Moreover $\varphi$ is differentiable almost everywhere in $\Omega$ and 
$\varphi\in W^1_{p,\loc}(\Omega)$.
\end{thm}

\begin{proof}
Fix an arbitrary cube $P$, $\overline{P}\subset \Omega$ with edges parallel to coordinate axes. We prove that $\varphi$ is absolutely continuous on almost all intersections of $P$ with lines parallel to the axis $x_n$.
Let $P_0$ be the orthogonal projection of $P$ on subspace $\{x_n=0\}=\mathbb R^{n-1}$ and $I$ be the orthogonal projection of $P$ on the axis $x_n$.
Then $P=P_0\times I$.

Since $\varphi$ is the homeomorphism then the Lebesgue measure $\Psi(E)=|\varphi(E)|$ induces by the rule $\Psi(A,P)=\Psi(A\times I)$ the monotone countable-additive function defined on measurable subsets of $P_0$.
By the Lebesgue theorem on differentiability (see, for example, \cite{RR55}) the upper $(n-1)$-dimensional volume derivative
$$
\overline{\Psi^{\prime}}(z,P)=\limsup\limits_{r\to 0}\frac{\Psi(B^{n-1}(z,r),P)}{r^{n-1}}<\infty,
$$
for almost all points $z\in P_0$. Here $B^{n-1}(z,r)$ is $(n-1)$-dimensional ball with center at $z\in P_0$ and radius $r$.

Fix a such point $z\in P_0$. Denote by $I_z=\{z\}\times I$ and $F$ be a compact subset of $I_z$.
By the condition of the theorem $F$ is a union of the following increasing sequence of closed sets
$$
F_k=\left\{x\in F : \frac{L^p_{\varphi}(x,r)r^{n-p}}{|\varphi(B(x,\lambda r))|}\leq \widetilde{H}_p^{\lambda},\,\,\,\text{for all}\,\,\, r<\frac{1}{k}\right\},
$$
with some constant $\widetilde{H}_p^{\lambda}>H_p^{\lambda}$.

Fix numbers
$k\in \mathbb N$, $\varepsilon>0$ and $t>0$. By Lemma \ref{lem:VaCovering} there exists a number
$\delta >0$
such that for any $r$, $0<r<\min(\delta,\frac{1}{k})$ there exists a collection $x_i\in F_k$, $i=1,2,...,N$, such that balls $B_i=B(x_i,r)$ cover $F_k$ and moreover each point of $F_k$ is contained in at most two balls, 
$$
Nr\leq H^1(F_k)+\varepsilon\,\,\text{and}\,\, |\varphi(x_i)-\varphi(y)|<t,\,\, y\in B(x_i,r).
$$ 

The balls $B(\varphi(x_i),L_{\varphi}(x_i,r))$ covering the image $\varphi(F_k)$.
Then, since for every ball its diameter $\diam(\varphi(B(x_i,r)))<t$, we have 
$$
H_{t}^{1}(\varphi(F_k))\leq \sum\limits_{i=1}^{N}L_{\varphi}(x_i,r).
$$
Hence, using the H\"older inequality we obtain
$$
\left(H_{t}^{1}(\varphi(F_k))\right)^p\leq N^{p-1}\sum\limits_{i=1}^{N}\left(L_{\varphi}(x_i,r)\right)^p.
$$
By the definition of the set $F_k$ we have
$$
\left(L_{\varphi}(x_i,r)\right)^p\leq
(\widetilde{H}_{p}^{\lambda})\cdot\frac{|\varphi(B(x_i,\lambda r))|}{r^{n-p}}.
$$
Hence
$$
\left(H_{t}^{1}(\varphi(F_k))\right)^p\leq
N^{p-1}\sum\limits_{i=1}^{N}\left(L_{\varphi}(x_i,r)\right)^p\\
\leq
(\widetilde{H}_{p}^{\lambda})\cdot(Nr)^{p-1}
\frac{\sum\limits_{i=1}^N |\varphi(B(x_i,\lambda r))|}{r^{n-1}}.
$$
Since any point of the set $\varphi(F_k)$ belongs no more than two sets $\varphi(B(x_i,r))$, $i=1,2,...,N$,
then
$$
\left(H_{t}^{1}(\varphi(F_k))\right)^p
\leq 2\lambda^{n-1}\left(H^1(F_k)+\varepsilon\right)^{p-1}
(\widetilde{H}_{p}^{\lambda})\cdot\frac{\Psi(B^{n-1}(z,2r),P)}{(2\lambda r)^{n-1}}.
$$
Passing to the limit while $r\to 0$, and turn to zero $\varepsilon$ and $t$ we obtain
$$
\left(H^{1}(\varphi(F_k))\right)^p
\leq 2\lambda^{n-1}(\widetilde{H}_{p}^{\lambda})\cdot (H^1(F_k))^{p-1}\Psi^{\prime}(z).
$$
Since $\varphi(F)$ is the limit of increasing sequence of compact sets $\varphi(F_k)$,
then
$$
H^{1}(\varphi(F))=\lim\limits_{k\to\infty} H^{1}(\varphi(F_k))
$$
and 
$$
\left(H^{1}(\varphi(F))\right)^p
\leq 2\lambda^{n-1}(\widetilde{H}_{p}^{\lambda})\cdot (H^1(F))^{p-1}\Psi^{\prime}(z)
$$
for almost all $z\in P_0$.  Hence $\varphi\in \ACL(\Omega)$.

Now we prove that $\varphi$ is differentiable almost everywhere in $\Omega$.
For all $r<\varepsilon(x)$ the inequality
$$
\biggl(\frac{L_{\varphi}(x,r)}{r}\biggr)^{p}\leq
(\widetilde{H}_{p}^{\lambda}) \frac{|\varphi(B(x,\lambda r))|}{r^n}
$$
holds with some constant $(\widetilde{H}_{p}^{\lambda})>{H}_{p}^{\lambda}$. Passing to the limit while $r\to 0$,
we obtain
$$
\limsup\limits_{r\to 0}\biggl(\frac{L_{\varphi}(x,r)}{r}\biggr)^{p}\leq \omega_n\lambda^n(\widetilde{H}_{p}^{\lambda}){\Psi}^{\prime}(x)<\infty,
$$
for almost all $x\in\Omega$. By the Stepanov theorem \cite{Fe69} we obtain that $\varphi$ is differentiable almost everywhere in $\Omega$.
Since ${\Psi}^{\prime}\in L_{1,\loc}(\Omega)$, then $|D\varphi|\in L_{p,\loc}(\Omega)$ and so $\varphi\in W^1_{p,\loc}(\Omega)$ \cite{M}.
\end{proof}

From Theorem~\ref{thm:ACLPPconf} follows

\begin{thm}
\label{thm:pp} Let $\varphi :\Omega\to \widetilde{\Omega}$ be a homeomorphism satisfy
$$
\limsup\limits_{r\to 0} H_{p}^{\lambda}(x,r)\leq H_p^{\lambda}<\infty, \,\,\text{for all}\,\,x\in\Omega,\,\,1<p<\infty.
$$
Then $\varphi$ generates a bounded composition operator 
$$
\varphi^{\ast}: L^1_p(\widetilde{\Omega})\to L^1_p(\Omega).
$$
\end{thm}

\begin{proof}
By Theorem~\ref{thm:ACLPPconf} the homeomorphism $\varphi$ belongs to the space $W^1_{p,\loc}(\Omega)$ and is differentiable almost everywhere in $\Omega$. Hence
$$
\lim\limits_{r\to 0}\frac{L_{\varphi}^{p}(x,r)}{r^p}=|D\varphi(x)|^p,\,\,\,\text{for almost all}\,\,\,x\in\Omega,
$$
and
$$
\lim\limits_{r\to 0}\frac{|\varphi(B(x,\lambda r))|}{r^n}
=\omega_n\lambda^n|J(x,\varphi)|,\,\,\,\text{for almost all}\,\,\,x\in\Omega.
$$
Hence
$$
|D\varphi(x)|^p\leq \omega_n\lambda^n H_p^{\lambda} |J(x,\varphi)| \,\,\,\text{a.~e. in}\,\,\, \Omega.
$$
Therefore by Theorem~\ref{CompTh}, we have that $\varphi$ generates a bounded composition operator 
$$
\varphi^{\ast}: L^1_p(\widetilde{\Omega})\to L^1_p(\Omega).
$$
\end{proof}

The inverse assertion is correct only under additional assumptions on $p$. Remark that in the case $n<p<\infty$ we can take $\lambda=1$.

\begin{thm}
\label{thm:pp-1}
Let a homeomorphism $\varphi :\Omega\to \widetilde{\Omega}$
generates a bounded composition operator 
$$
\varphi^{\ast}: L^1_p(\widetilde{\Omega})\to L^1_p(\Omega),\,\, n<p< \infty.
$$
Then there exists a constant $H_p^{1}<\infty$ such that
$$
\limsup\limits_{r\to 0} H_{p}^{1}(x_0,r)=\limsup\limits_{r\to 0}\frac{L_{\varphi}^{p}(x_0,r)r^{n-p}}{|\varphi(B(x_0,r))|}\leq H_p^{1}<\infty, \,\,\text{for all}\,\,x_0\in\Omega.
$$
\end{thm}

\begin{proof} Fix a point $x_0\in\Omega$ and $r>0$ such that $B(x_0,2r)\subset\Omega$. In the domain $\widetilde{\Omega}$ we consider a condenser $(F_0,F_1)\subset \widetilde{\Omega}$, where
$$
F_0=\widetilde{\Omega}\setminus \varphi(B(x_0,r)), \,\,F_1=\{\varphi(x_0)\}.
$$
Since the homeomorphism $\varphi :\Omega\to \widetilde{\Omega}$ generates a bounded composition operator 
$$
\varphi^{\ast}: L^1_p(\widetilde{\Omega})\to L^1_p(\Omega),
$$
then by \cite{GGR95,U93}
$$
\cp_{p}^{\frac{1}{p}}(\varphi^{-1}(F_0),\varphi^{-1}(F_1);\Omega) \leq K_{p,p}(\varphi;\Omega)\cp_{p}^{\frac{1}{p}}(F_0,F_1;\widetilde{\Omega}).
$$
By capacity estimates \cite{GResh}
\begin{multline*}
c(n,p)r^{n-p}=\cp_{p}(\varphi^{-1}(F_0),\varphi^{-1}(F_1);\Omega) \\
\leq K^p_{p,p}(\varphi;\Omega)\cp_{p}(F_0,F_1;\widetilde{\Omega})\leq K^p_{p,p}(\varphi;\Omega) \frac{|\varphi(B(x_0,r)|}{L^p_{\varphi}(x_0,r)}.
\end{multline*}

Hence
$$
\frac{L^p_{\varphi}(x_0,r) r^{n-p}}{|\varphi(B(x_0,r)|}\leq \frac{K^p_{p,p}(\varphi;\Omega)}{c(n,p)}.
$$
Setting $H_p^{1}:=c^{-1}(n,p)\cdot K^p_{p,p}(\varphi;\Omega)$ we obtain
$$
\limsup\limits_{r\to 0} H_{p}^{1}(x_0,r)=\limsup\limits_{r\to 0}\frac{L_{\varphi}^{p}(x_0,r)r^{n-p}}{|\varphi(B(x_0,r))|}\leq H_p^{1}<\infty, \,\,\text{for all}\,\,x_0\in\Omega.
$$
\end{proof}

In the case $n-1<p<n$ we use the Teichm\"uller type capacity estimates and so we should take $\lambda>1$.

\begin{thm}
\label{thm:pp-lambda}
Let a homeomorphism $\varphi :\Omega\to \widetilde{\Omega}$
generates a bounded composition operator 
$$
\varphi^{\ast}: L^1_p(\widetilde{\Omega})\to L^1_p(\Omega),\,\, n-1<p<n.
$$
Then there exists a constant $H_p^{\lambda}<\infty$, $\lambda>1$, such that
$$
\limsup\limits_{r\to 0} H_{p}^{\lambda}(x_0,r)=\limsup\limits_{r\to 0}\frac{L_{\varphi}^{p}(x_0,r)r^{n-p}}{|\varphi(B(x_0,\lambda r))|}\leq H_p^{\lambda}<\infty, \,\,\text{for all}\,\,x_0\in\Omega.
$$
\end{thm}

\begin{proof} Fix a point $x_0\in\Omega$ and $r>0$ such that $B(x_0,2 \lambda r)\subset\Omega$. Denote by $y_0:=\varphi(x_0)$.
Let a point $y_1\in f(S(x_0,r))$ is chosen such that $L_{\varphi}(x_0,r)=\rho(y_0,y_1)$.
By symbol $y_2$ we denote the second point of the intersection of the line, passing  through
points $y_1$ and $y_0$, with the set $f(S(x_0,r))$.
In the domain $\widetilde{\Omega}$ we consider continuums
\begin{align}
F_0&=\{y\in \widetilde{\Omega}: |y-y_2|\leq |y_2-y_0|\}
\cap f(B(x_0,\lambda r)),
\nonumber
\\
F_1&=\{y\in \widetilde{\Omega}: |y-y_2|\geq |y_2-y_1|\}
\cap f(B(x_0,\lambda r)).
\nonumber
\end{align}
Since $\varphi$ generates a bounded composition operator 
$$
\varphi^{\ast}: L^1_p(\widetilde{\Omega})\to L^1_p(\Omega),
$$
then by \cite{GGR95,U93}
$$
\cp_{p}^{\frac{1}{p}}(\varphi^{-1}(F_0),\varphi^{-1}(F_1);\Omega) \leq K_{p,p}(\varphi;\Omega)\cp_{p}^{\frac{1}{p}}(F_0,F_1;\widetilde{\Omega}).
$$
By capacity estimates \cite{GResh,Vu}
\begin{multline}
c(n,p,\lambda)r^{n-p}\leq
{\cp}_{p}(\varphi^{-1}(F_0),\varphi^{-1}(F_1);\Omega)\\
\leq K^p_{p,p}(\varphi;\Omega)\cp_{p}(F_0,F_1;\widetilde{\Omega})\leq K^p_{p,p}(\varphi;\Omega) \frac{|\varphi(B(x_0,\lambda r)|}{L^p_{\varphi}(x_0,r)}.
\nonumber
\end{multline}
Hence
$$
\frac {L_{\varphi}^{p}(x_0,r)r^{n-p}}{|\varphi(B(x_0,\lambda r))|}\leq \frac{K^p_{p,p}(\varphi;\Omega)}{c(n,p,\lambda)}.
$$
Setting $H_p^{\lambda}:=c^{-1}(n,p,\lambda)\cdot K^p_{p,p}(\varphi;\Omega)$ and passing the the limit while $r\to 0$, we obtain that
$$
\limsup\limits_{r\to 0} H_{p}^{\lambda}(x_0,r)=\limsup\limits_{r\to 0}\frac{L_{\varphi}^{p}(x_0,r)r^{n-p}}{|\varphi(B(x_0,\lambda r))|}\leq H_p^{\lambda}<\infty, \,\,\text{for all}\,\,x_0\in\Omega.
$$
\end{proof}

\subsection{Weak $(p,q)$-quasiconformal mappings}

Let us recall the notion of the set function $\widetilde{\Phi}_{p,q}(\widetilde A)$, defined on open bounded subsets $\widetilde A\subset\widetilde{\Omega}$ and associated with the composition operator $\varphi^\ast: L^1_p(\widetilde{\Omega})\to L^1_{q}(\Omega)$, $1 < q < p < \infty$:

\begin{equation}\label{setfunc}
\widetilde{\Phi}_{p,q}(\widetilde{A})=\sup\limits_{f\in L^1_p(\widetilde{A})\cap C_0(\widetilde{A})}\left(\frac{\|\varphi^{\ast}(f)\mid L^1_q(\Omega)\|}{\|f\mid L^1_p(\widetilde{A})\|}\right)^{\kappa}, \quad 1/{\kappa}=1/q-1/p.
\end{equation}

\begin{thm}\cite{U93}
\label{thmsetfunc}
\label{CompPhi} Let a homeomorphism $\varphi:\Omega\to\widetilde{\Omega}$
between two domains $\Omega$ and $\widetilde{\Omega}$ generates a bounded composition
operator 
\[
\varphi^{\ast}:L^1_p(\widetilde{\Omega})\to L^1_q(\Omega),\,\,\,1\leq q< p\leq\infty.
\]
Then the function $\widetilde{\Phi}_{p,q}(\widetilde{A})$, defined by (\ref{setfunc}), is a bounded monotone countably additive set function defined on open bounded subsets $\widetilde{A}\subset\widetilde{\Omega}$.
\end{thm}

Recall that a nonnegative mapping $\Phi$ defined on open subsets of $\Omega$ is called a monotone countably additive set function \cite{RR55,VU04} if

\noindent
1) $\Phi(U_1)\leq \Phi(U_2)$ if $U_1\subset U_2\subset\Omega$;

\noindent
2)  for any collection $U_i \subset U \subset \Omega$, $i=1,2,...$, of mutually disjoint open sets
$$
\sum_{i=1}^{\infty}\Phi(U_i) = \Phi\left(\bigcup_{i=1}^{\infty}U_i\right).
$$

The following lemma gives properties of monotone countably additive set functions defined on open subsets of $\Omega\subset \mathbb R^n$ \cite{RR55,VU04}.

\begin{lem}
\label{lem:AddFun}
Let $\Phi$ be a monotone countably additive set function defined on open subsets of the domain $\Omega\subset \mathbb R^n$. Then

\noindent
(a) at almost all points $x\in \Omega$ there exists a finite derivative
$$
\lim\limits_{r\to 0}\frac{\Phi(B(x,r))}{|B(x,r)|}=\Phi'(x);
$$

\noindent
(b) $\Phi'(x)$ is a measurable function;

\noindent
(c) for every open set $U\subset \Omega$ the inequality
$$
\int\limits_U\Phi'(x)~dx\leq \Phi(U)
$$
holds.
\end{lem}

Let $\varphi :\Omega\to \widetilde{\Omega}$ be a homeomorphism. Follow \cite{U94} we introduce the geometric $(p,q)$-dilatation
$$
H_{p,q}^{\lambda}(x,r;\Phi_{p,q})=\frac{L_{\varphi}^{p}(x,r)r^{n-p}}{|\varphi(B(x,\lambda r))|}\biggl(\frac{|B(x,r)|}{{\Phi_{p,q}}(B(x,\lambda r))}\biggr)^{\frac{p-q}{q}},\,\, \lambda\geq 1,
$$
where $L_{\varphi}(x,r)=\max\limits_{|x-y|=r}|\varphi(x)-\varphi(y)|$ and $\Phi_{p,q}$ a bounded monotone countable-additive absolutely continuous set function defined on open subsets of $\Omega$.

Remark that geometric characteristics of mappings with an integrable quasiconformal distortion were considered in another terms in \cite{KM02,KR05}.

\begin{thm} 
\label{thm:ACLPQconf} 
Let $\varphi :\Omega\to \widetilde{\Omega}$ be a homeomorphism which satisfies
$$
\limsup\limits_{r\to 0} H_{p,q}^{\lambda}(x,r;\Phi_{p,q})\leq H_{p,q}^{\lambda}(\Phi_{p,q})<\infty, \,\,\text{for all}\,\,x\in\Omega,\,\,1<q<p<\infty.
$$
Then the homeomorphism $\varphi$ belongs to $\ACL(\Omega)$. Moreover $\varphi$ is differentiable almost everywhere in $\Omega$ and 
$\varphi\in W^1_{q,\loc}(\Omega)$.
\end{thm}

\begin{proof}
Fix an arbitrary cube $P$, $\overline{P}\subset \Omega$ with edges parallel to coordinate axes. We prove that $\varphi$ is absolutely continuous on almost all intersections of $P$ with lines parallel to the axis $x_n$.
Let $P_0$ be the orthogonal projection of $P$ on subspace $\{x_n=0\}=\mathbb R^{n-1}$ and $I$ be the orthogonal projection of $P$ on the axis $x_n$.
Then $P=P_0\times I$.

Since $\varphi$ is the homeomorphism then the Lebesgue measure $\Psi(E)=|\varphi(E)|$ induces by the rule $\Psi(A,P)=\Psi(A\times I)$ the monotone countable-additive function defined on measurable subsets of $P_0$.
By the Lebesgue theorem on differentiability (see, for example, \cite{RR55} the upper $(n-1)$-dimensional volume derivative
$$
\overline{\Psi^{\prime}}(z,P)=\limsup\limits_{r\to 0}\frac{\Psi(B^{n-1}(z,r),P)}{r^{n-1}}
$$
is finite for almost all points $z\in P_0$. Here $B^{n-1}(z,r)$ is $(n-1)$-dimensional ball with center at $z\in P_0$ and radius $r$.

Since $\Phi_{p,q}$ is a bounded monotone countable-additive absolutely continuous set function, then $\Phi_{p,q}$ can be extended on measurable sets $E\subset\Omega$, setting
$$
\Phi_{p,q}(E)=\inf\limits_A \Phi_{p,q}(A), \,\,E\subset A\subset\Omega,
$$
where $A$ is an open set.
This monotone countable-additive function $\Phi_{p,q}$ induces by the rule $\Phi_{p,q}(A,P)=\Phi_{p,q}(A\times I)$ the monotone countable-additive function defined on measurable subsets of $P_0$. By the Lebesgue theorem on differentiability (see, for example, \cite{RR55} the upper $(n-1)$-dimensional volume derivative
$$
\overline{{\Phi}_{p,q}^{\prime}}(z,P)=\limsup\limits_{r\to 0}\frac{{\Phi_{p,q}}(B^{n-1}(z,r),P)}{r^{n-1}}
$$
is also finite for almost all points $z\in P_0$. 

Fix a such point $z\in P_0$ in which $\overline{\Psi^{\prime}}(z,P)<\infty$ and $\overline{{\Phi}_{p,q}^{\prime}}(z,P)<\infty$. Let $I_z=\{z\}\times I$ and $F$ be a compact subset of $I_z$.
By the condition of the theorem the set $F$ is a union of the following increasing sequence of closed sets
$$
F_k=\left\{x\in F : \frac{L^p_{\varphi}(x,r)r^{n-p}}{|\varphi(B(x,\lambda r))|}\leq \widetilde{H}_{p,q}^{\lambda}(\Phi_{p,q})\biggl(\frac{\Phi_{p,q}(B(x,\lambda r))}{r^n}\biggr)^{\frac{p-q}{q}},\,\,\,\text{for all}\,\,\, r<\frac{1}{k}\right\},
$$
where a constant $\widetilde{H}_{p,q}^{\lambda}(\Phi_{p,q})>H_{p,q}^{\lambda}(\Phi_{p,q})$. Note, that closeness of sets $F_k$ follows from the absolute continuity of the set function $\Phi_{p,q}$.

Fix numbers $k$, $\varepsilon>0$ and $t>0$. By Lemma \ref{lem:VaCovering} there exists a number
$\delta >0$
such that for any $r$, $0<r<\min(\delta,\frac{1}{k})$
there exists a sequence $x_i\in F_k$, $i=1,2,...,N$, such that balls $B_i=B(x_i,r)$ cover $F_k$ (moreover each point of $F_k$ is contained in at most two balls), 
$$
Nr\leq H^1(F_k)+\varepsilon \,\, \text{and}\,\, |\varphi(x_i)-\varphi(y)|<t, \, y\in B(x_i,r).
$$ 

The balls $B(\varphi(x_i),L_{\varphi}(x_i,r))$ covering the image $\varphi(F_k)$.
Then, since for every ball its diameter $\diam(\varphi(B(x_i,r)))<t$, we have that
$$
H_{t}^{1}(\varphi(F_k))\leq \sum\limits_{i=1}^{N}L_{\varphi}(x_i,r).
$$
Using the H\"older inequality we obtain
$$
\left(H_{t}^{1}(\varphi(F_k))\right)^q\leq
N^{q-1}\sum\limits_{i=1}^{N}\left(L_{\varphi}(x_i,r)\right)^q.
$$
So, by the definition of the set $F_k$ we have
$$
\left(L_{\varphi}(x_i,r)\right)^q\leq  \left(\widetilde{H}_{p,q}^{\lambda}(\Phi_{p,q})\right)^{\frac{q}{p}}\frac{|\varphi(B(x_i,\lambda r))|^{\frac{q}{p}}\left({\Phi_{p,q}}(B(x_i,\lambda r))\right)^{\frac{p-q}{p}}}{r^{n-q}}.
$$
Hence
\begin{multline}
\left(H_{t}^{1}(\varphi(F_k))\right)^q\leq
N^{q-1}\sum\limits_{i=1}^{N}\left(L_{\varphi}(x_i,r)\right)^q\\
\leq \left(\widetilde{H}_{p,q}^{\lambda}(\Phi_{p,q})\right)^{\frac{q}{p}}N^{q-1}\sum\limits_{i=1}^{N}\frac{|\varphi(B(x_i,\lambda r))|^{\frac{q}{p}}\left({\Phi_{p,q}}(B(x_i,\lambda r))\right)^{\frac{p-q}{p}}}{r^{n-q}}\\
\leq \left(\widetilde{H}_{p,q}^{\lambda}(\Phi_{p,q})\right)^{\frac{q}{p}}(Nr)^{q-1}\left(\frac{\sum\limits_{i=1}^{N}|\varphi(B(x_i,\lambda r))|}{r^{n-1}}\right)^{\frac{q}{p}}
\left(\frac{\sum\limits_{i=1}^{N}\tilde{\Phi}(B_{p,q}(x_i,\lambda r))}{r^{n-1}}\right)^{\frac{p-q}{p}}.
\nonumber
\end{multline}
Since any point of the set $\varphi(F_k)$ belongs no more than two sets $\varphi(B(x_i,r))$, $i=1,2,...,N$,
then
\begin{multline*}
\left(H_{t}^{1}(\varphi(F_k))\right)^q \leq 
c(p,q,\lambda)\left(H^1(F_k)+\varepsilon\right)^{q-1}
\left(\widetilde{H}_{p,q}^{\lambda}(\Phi_{p,q})\right)^{\frac{q}{p}}
\\
\times \left(\frac{\Psi(B^{n-1}(z,\lambda r),P)}{(\lambda r)^{n-1}}\right)^{\frac{q}{p}}\left(\frac{{\Phi_{p,q}}(B^{n-1}(z,\lambda r),P)}{(\lambda r)^{n-1}}\right)^{\frac{p-q}{p}},
\end{multline*}
where a constant $c(p,q,\lambda)$ depends on $p$, $q$ and $\lambda$.

Passing to the limit while $r\to 0$, and turn to zero $\varepsilon$ and $t$ we obtain
$$
\left(H^{1}(\varphi(F_k))\right)^q
\leq c(p,q,\lambda)
\left(\widetilde{H}_{p,q}^{\lambda}(\Phi_{p,q})\right)^{\frac{q}{p}} \left(H^1(F_k)\right)^{q-1}\left(\overline{\Psi^{\prime}}(z,P)\right)^{\frac{q}{p}}\left(\overline{{\Phi}_{p,q}^{\prime}}(z,P)\right)^{\frac{p-q}{p}}.
$$
Since $\varphi(F)$ is the limit of increasing sequence of compact sets $\varphi(F_k)$,
then
$$
H^{1}(\varphi(F))=\lim\limits_{k\to\infty} H^{1}(\varphi(F_k)).
$$
Hence
$$
\left(H^{1}(\varphi(F))\right)^q
\leq c(p,q,\lambda)
\left(\widetilde{H}_{p,q}^{\lambda}(\Phi_{p,q})\right)^{\frac{q}{p}} \left(H^1(F)\right)^{q-1}\left(\overline{\Psi^{\prime}}(z,P)\right)^{\frac{q}{p}}\left(\overline{{\Phi}_{p,q}^{\prime}}(z,P)\right)^{\frac{p-q}{p}},
$$
and therefore $\varphi\in \ACL(\Omega)$.

Now we prove that
$\varphi$
is differentiable almost everywhere in $\Omega$. For all $r<\varepsilon(x)$ the inequality
$$
\biggl(\frac{L_{\varphi}(x,r)}{r}\biggr)^{p}\leq  \widetilde{H}_{p,q}^{\lambda}(\Phi_{p,q}) \frac{|\varphi(B(x,\lambda r))|}{r^n}\left(\frac{{\Phi_{p,q}}(B(x,\lambda r))}{r^n}\right)^{\frac{p-q}{q}}
$$
holds with some constant  $\widetilde{H}_{p,q}^{\lambda}(\Phi_{p,q})>H_{p,q}^{\lambda}(\Phi_{p,q})$. Passing to the limit while $r\to 0$,
we obtain that the inequality 
$$
\limsup\limits_{r\to 0} \biggl(\frac{L_{\varphi}(x,r)}{r}\biggr)^{p}\leq c(p,q,\lambda) \widetilde{H}_{p,q}^{\lambda}(\Phi_{p,q}){\Psi}^{\prime}(x)({\Phi_{p,q}}'(x))^{\frac{p-q}{q}}<\infty
$$
holds for almost all $x\in\Omega$ with some constant $c(p,q,\lambda)$.
By the Stepanov theorem \cite{Fe69} we obtain that $\varphi$ is differentiable almost everywhere in $\Omega$.

Hence
$$
|D\varphi(x)|^q\leq \left(c(p,q,\lambda) \widetilde{H}_{p,q}^{\lambda}(\Phi_{p,q})\right)^{\frac{q}{p}}\left({\Psi}^{\prime}(x)\right)^{\frac{q}{p}}(\Phi_{p,q}^{\prime}(x))^{\frac{p-q}{p}}. 
$$
So, for any bounded open set $U\subset\Omega$, $\overline{U}\subset\Omega$,  by using the H\"older inequality we have
\begin{multline*}
\int\limits_{U}|D\varphi(x)|^q~dx\leq \left(c(p,q,\lambda) \widetilde{H}_{p,q}^{\lambda}(\Phi_{p,q})\right)^{\frac{q}{p}}
\int\limits_{U}\left({\Psi}^{\prime}(x)\right)^{\frac{q}{p}}(\Phi_{p,q}^{\prime}(x))^{\frac{p-q}{p}}~dx\\
\leq \left(c(p,q,\lambda) \widetilde{H}_{p,q}^{\lambda}(\Phi_{p,q})\right)^{\frac{q}{p}}
\left(\int\limits_{U}{\Psi}^{\prime}(x)~dx\right)^{\frac{q}{p}}\left(\int\limits_{U}\Phi_{p,q}^{\prime}(x)~dx\right)^{\frac{p-q}{p}}<\infty. 
\end{multline*}
Therefore $|D\varphi|\in L_{q,\loc}(\Omega)$ and we have that $\varphi\in W^1_{q,\loc}(\Omega)$ \cite{M}.
\end{proof}

From Theorem~\ref{thm:ACLPQconf} follows

\begin{thm}
\label{thm:pq}
Let $\varphi :\Omega\to \widetilde{\Omega}$ be a homeomorphism which satisfies
$$
\limsup\limits_{r\to 0} H_{p,q}^{\lambda}(x,r;\Phi_{p,q})\leq H_{p,q}^{\lambda}(\Phi_{p,q})<\infty, \,\,\text{for all}\,\,x\in\Omega,\,\,1<q<p<\infty.
$$
Then $\varphi$ generates a bounded composition operator 
$$
\varphi^{\ast}: L^1_p(\widetilde{\Omega})\to L^1_q(\Omega).
$$
\end{thm}

\begin{proof}
By Theorem~\ref{thm:ACLPQconf} the homeomorphism $\varphi$ belongs to the space $W^1_{q,\loc}(\Omega)$ and is differentiable almost everywhere in $\Omega$. Hence
$$
\lim\limits_{r\to 0}\frac{L_{\varphi}^{p}(x,r)}{r^p}=|D\varphi(x)|^p,\,\,\,\text{for almost all}\,\,\,x\in\Omega,
$$
and
$$
\lim\limits_{r\to 0}\frac{|\varphi(B(x,\lambda r))|}{r^n}
=\omega_n\lambda^n|J(x,\varphi)|,\,\,\,\text{for almost all}\,\,\,x\in\Omega.
$$
So, we obtain
$$
|D\varphi(x)|^p\leq c(p,q,\lambda) \widetilde{H}_{p,q}^{\lambda}(\Phi_{p,q}) |J(x,\varphi)|\Phi^{\prime}_{p,q}(x)^{\frac{p-q}{q}} \,\,\,\text{for almost all}\,\,\,x\in\Omega,
$$
and $\varphi$ is the mapping of finite distortion. 

Hence
$$
\left(\frac{|D\varphi(x)|^p}{|J(x,\varphi)|}\right)^{\frac{q}{p-q}}=\lim\limits_{r\to 0}\left(\frac{L^p_{\varphi}(x,r)r^{n-p}}{|\varphi(B(x,r))|}\right)^{\frac{q}{p-q}}
\leq
c(p,q,\lambda) \widetilde{H}_{p,q}^{\lambda}(\Phi_{p,q})\Phi^{\prime}_{p,q}(x), \,\,\,\text{a.~e. in}\,\,\, \Omega\setminus Z,
$$
where $Z=\{x\in\Omega : J(x,\varphi)=0\}$.

Integrating of the last inequality on an arbitrary open bounded subset $U\subset\Omega$ we obtain
\begin{multline*}
\int\limits_U\left(\frac{|D\varphi(x)|^p}{|J(x,\varphi)|}\right)^{\frac{q}{p-q}}~dx=\int\limits_{U\setminus Z}\left(\frac{|D\varphi(x)|^p}{|J(x,\varphi)|}\right)^{\frac{q}{p-q}}~dx
\\
\leq c(p,q,\lambda) \widetilde{H}_{p,q}^{\lambda}(\Phi_{p,q})\int\limits_{U\setminus Z} \Phi^{\prime}_{p,q}(x)~dx
\leq c(p,q,\lambda) \widetilde{H}_{p,q}^{\lambda}(\Phi_{p,q}) \Phi_{p,q}(U)
\\
\leq c(p,q,\lambda) \widetilde{H}_{p,q}^{\lambda}(\Phi_{p,q}) \Phi_{p,q}(\Omega)<\infty.
\end{multline*}
Since the choice of $U\subset\Omega$ is arbitrary, we have
$$
\int\limits_U\left(\frac{|D\varphi(x)|^p}{|J(x,\varphi)|}\right)^{\frac{q}{p-q}}~dx
\leq c(p,q,\lambda) \widetilde{H}_{p,q}^{\lambda}(\Phi_{p,q}) \Phi_{p,q}(\Omega)<\infty.
$$

Therefore by Theorem~\ref{CompTh}, we have that $\varphi$ generates a bounded composition operator 
$$
\varphi^{\ast}: L^1_p(\widetilde{\Omega})\to L^1_q(\Omega).
$$
\end{proof}

Let a homeomorphism $\varphi :\Omega\to \widetilde{\Omega}$
generates a bounded composition operator 
$$
\varphi^{\ast}: L^1_p(\widetilde{\Omega})\to L^1_q(\Omega),\,\, 1<q<p< \infty.
$$
We defined a bounded monotone countably additive set function $\Phi_{p,q}$ defined on open bounded subsets $A\subset{\Omega}$ by the rule
$$
\Phi_{p,q}(A)=\widetilde{\Phi}_{p,q}(\varphi(A)),
$$
where $\widetilde{\Phi}_{p,q}$ is defined by \eqref{setfunc}. By Theorem~\ref{CompTh}
$$
\Phi_{p,q}(A)\leq\int\limits_A \left(K_p(x)\right)^{\frac{pq}{p-q}}~dx,
$$
and so the set function $\Phi_{p,q}$ is absolutely continuous.

\begin{thm}
\label{thm:pq-1}
Let a homeomorphism $\varphi :\Omega\to \widetilde{\Omega}$
generates a bounded composition operator 
$$
\varphi^{\ast}: L^1_p(\widetilde{\Omega})\to L^1_p(\Omega),\,\, n<q<p< \infty.
$$
Then there exists a constant $H_{p,q}^{1}(\Phi_{p,q})<\infty$ such that
\begin{multline*}
\limsup\limits_{r\to 0} H_{p,q}^{1}(x_0,r;\Phi_{p,q})\\
=\limsup\limits_{r\to 0}\frac{L_{\varphi}^{p}(x_0,r)r^{n-p}}{|\varphi(B(x_0, r))|}\biggl(\frac{|B(x_0,r)|}{{\Phi_{p,q}}(B(x_0, r))}\biggr)^{\frac{p-q}{q}}\leq H_{p,q}^{1}(\Phi_{p,q})<\infty, \,\,\text{for all}\,\,x_0\in\Omega.
\end{multline*}
\end{thm}

\begin{proof} Fix a point $x_0\in\Omega$ and $r>0$ such that $B(x_0,2r)\subset\Omega$. In the domain $\widetilde{\Omega}$ we consider a condenser $(F_0,F_1)\subset \widetilde{\Omega}$, where
$$
F_0=\widetilde{\Omega}\setminus \varphi(B(x_0,r)), \,\,F_1=\{\varphi(x_0)\}.
$$
Since the homeomorphism $\varphi :\Omega\to \widetilde{\Omega}$ generates a bounded composition operator 
$$
\varphi^{\ast}: L^1_p(\widetilde{\Omega})\to L^1_q(\Omega),
$$
then by \cite{U93}
$$
\cp_q^{\frac{1}{q}}(\varphi^{-1}(F_0),\varphi^{-1}(F_1);\Omega)\leq \widetilde{\Phi}_{p,q}(\widetilde{\Omega}\setminus(F_0\cup F_1))^{\frac{p-q}{pq}}\cp_p^{\frac{1}{p}}(\varphi(F_0),\varphi(F_1);\widetilde{\Omega}).
$$
By capacity estimates \cite{GResh} and taking into account that 
$$
\widetilde{\Phi}_{p,q}(\widetilde{\Omega}\setminus(F_0\cup F_1))={\Phi}_{p,q}({\Omega}\setminus\varphi^{-1}(F_0\cup F_1)),
$$
we obtain
\begin{multline*}
c(n,q)r^{\frac{n-q}{q}}=\cp_{q}^{\frac{1}{q}}(\varphi^{-1}(F_0),\varphi^{-1}(F_1);\Omega) \\
\leq \left({\Phi}_{p,q}(B(x_0,r))\right)^{\frac{p-q}{pq}}\cp_{p}^{\frac{1}{p}}(F_0,F_1;\widetilde{\Omega})\leq \left({\Phi}_{p,q}(B(x_0,r))\right)^{\frac{p-q}{pq}} \frac{|\varphi(B(x_0,r)|^{\frac{1}{p}}}{L_{\varphi}(x_0,r)}.
\end{multline*}
Hence
$$
c(n,q)\frac{L^p_{\varphi}(x_0,r) r^{\frac{(n-q)p}{q}}}{|\varphi(B(x_0,r)|}\leq \left({\Phi}_{p,q}(B(x_0,r))\right)^{\frac{p-q}{q}},
$$
and so
$$
\frac{L^p_{\varphi}(x_0,r) r^{n-p}}{|\varphi(B(x_0,r)|}\leq c^{-1}(n,p,q) \left(\frac{{\Phi}_{p,q}(B(x_0,r))}{|B(x_0,r)|}\right)^{\frac{p-q}{q}}.
$$

Setting $H_{p,q}^{1}(\Phi_{p,q}):=c^{-1}(n,p,q)$ we obtain
\begin{multline*}
\limsup\limits_{r\to 0} H_{p,q}^{1}(x_0,r;\Phi_{p,q})\\
=\limsup\limits_{r\to 0}\frac{L_{\varphi}^{p}(x_0,r)r^{n-p}}{|\varphi(B(x_0, r))|}\biggl(\frac{|B(x_0,r)|}{{\Phi_{p,q}}(B(x_0, r))}\biggr)^{\frac{p-q}{q}}\leq H_{p,q}^{1}(\Phi_{p,q})<\infty, \,\,\text{for all}\,\,x_0\in\Omega.
\end{multline*}
\end{proof}

In the case $n-1<q<n$ we use the Teichm\"uller type capacity estimates and so we should take $\lambda>1$.

\begin{thm}
\label{thm:pq-lambda}
Let a homeomorphism $\varphi :\Omega\to \widetilde{\Omega}$
generates a bounded composition operator 
$$
\varphi^{\ast}: L^1_p(\widetilde{\Omega})\to L^1_p(\Omega),\,\, n-1<q<n,\, q<p< \infty.
$$
Then there exists a constant $H_{p,q}^{k}(\Phi_{p,q})<\infty$, $\lambda>1$, such that
\begin{multline*}
\limsup\limits_{r\to 0} H_{p,q}^{\lambda}(x_0,r;\Phi_{p,q})\\
=\limsup\limits_{r\to 0}\frac{L_{\varphi}^{p}(x_0,r)r^{n-p}}{|\varphi(B(x_0,\lambda r))|}\biggl(\frac{|B(x_0,r)|}{{\Phi_{p,q}}(B(x_0, \lambda r))}\biggr)^{\frac{p-q}{q}}\leq H_{p,q}^{\lambda}(\Phi_{p,q})<\infty, \,\,\text{for all}\,\,x_0\in\Omega.
\end{multline*}
\end{thm}

\begin{proof}
 Fix a point $x_0\in\Omega$ and $r>0$ such that $B(x_0,2 \lambda r)\subset\Omega$. Denote by $y_0:=\varphi(x_0)$.
Let a point $y_1\in f(S(x_0,r))$ is chosen such that $L_{\varphi}(x_0,r)=\rho(y_0,y_1)$.
By symbol $y_2$ we denote the second point of the intersection of the line, passing  through
points $y_1$ and $y_0$, with the set $f(S(x_0,r))$.
In the domain $\widetilde{\Omega}$ we consider continuums
\begin{align}
F_0&=\{y\in \widetilde{\Omega}: |y-y_2|\leq |y_2-y_0|\}
\cap f(B(x_0,\lambda r)),
\nonumber
\\
F_1&=\{y\in \widetilde{\Omega}: |y-y_2|\geq |y_2-y_1|\}
\cap f(B(x_0,\lambda r)).
\nonumber
\end{align}
Since $\varphi$ generates a bounded composition operator 
$$
\varphi^{\ast}: L^1_p(\widetilde{\Omega})\to L^1_q(\Omega),
$$
then by \cite{U93}
$$
\cp_q^{\frac{1}{q}}(\varphi^{-1}(F_0),\varphi^{-1}(F_1);\Omega)\leq \widetilde{\Phi}_{p,q}(\widetilde{\Omega}\setminus(F_0\cup F_1))^{\frac{p-q}{pq}}\cp_p^{\frac{1}{p}}(\varphi(F_0),\varphi(F_1);\widetilde{\Omega}).
$$
By capacity estimates \cite{GResh} and taking into account that 
$$
\widetilde{\Phi}_{p,q}(\widetilde{\Omega}\setminus(F_0\cup F_1))={\Phi}_{p,q}({\Omega}\setminus\varphi^{-1}(F_0\cup F_1)),
$$
we obtain
\begin{multline*}
c(n,q)r^{\frac{n-q}{q}}=\cp_{q}^{\frac{1}{q}}(\varphi^{-1}(F_0),\varphi^{-1}(F_1);\Omega) \\
\leq \left({\Phi}_{p,q}(B(x_0,\lambda r))\right)^{\frac{p-q}{pq}}\cp_{p}^{\frac{1}{p}}(F_0,F_1;\widetilde{\Omega})\leq \left({\Phi}_{p,q}(B(x_0,\lambda r))\right)^{\frac{p-q}{pq}} \frac{|\varphi(B(x_0,\lambda r)|^{\frac{1}{p}}}{L_{\varphi}(x_0,r)}.
\end{multline*}
Hence
$$
c(n,q)\frac{L^p_{\varphi}(x_0,r) r^{\frac{(n-q)p}{q}}}{|\varphi(B(x_0,r)|}\leq \left({\Phi}_{p,q}(B(x_0,r))\right)^{\frac{p-q}{q}},
$$
and so
$$
\frac{L^p_{\varphi}(x_0,r) r^{n-p}}{|\varphi(B(x_0,\lambda r)|}\leq c^{-1}(n,p,q) \left(\frac{{\Phi}_{p,q}(B(x_0,\lambda r))}{|B(x_0,r)|}\right)^{\frac{p-q}{q}}.
$$

Setting $H_{p,q}^{\lambda }(\Phi_{p,q}):=c^{-1}(n,p,q,\lambda)$ we obtain
\begin{multline*}
\limsup\limits_{r\to 0} H_{p,q}^{\lambda }(x_0,r;\Phi_{p,q})\\
=\limsup\limits_{r\to 0}\frac{L_{\varphi}^{p}(x_0,r)r^{n-p}}{|\varphi(B(x_0,\lambda r))|}\biggl(\frac{|B(x_0,r)|}{{\Phi_{p,q}}(B(x_0,\lambda  r))}\biggr)^{\frac{p-q}{q}}\leq H_{p,q}^{\lambda }(\Phi_{p,q})<\infty, \,\,\text{for all}\,\,x_0\in\Omega.
\end{multline*}
\end{proof}

\vskip 0.5cm

Alexander Ukhlov; Department of Mathematics, Ben-Gurion University of the Negev, P.O.Box 653, Beer Sheva, 8410501, Israel 
							
\emph{E-mail address:} \email{ukhlov@math.bgu.ac.il


\medskip

\begin{thebibliography}{References}

\bibitem{Fe69} H.~Federer, Geometric measure theory, Sp\-rin\-ger Verlag, Berlin, (1969).

\bibitem{Ge60} F.~W.~Gehring, The definitions and exceptional sets for quasiconformal mappings, Ann. Acad. Sci. Fenn. Ser. AI, 281 (1960), 
1--28. 

\bibitem{Ger69} F.~W.~Gehring, Lipschitz mappings and the $p$-capacity of rings in $n$-space, Advances in the theory of Riemann surfaces (Proc. Conf., Stony Brook, N. Y., 1969), 175--193. Ann. of Math. Studies, No. 66. Princeton Univ. Press, Princeton, N. J. (1971)

\bibitem{GG94} V.~Gol'dshtein, L.~Gurov, Applications of change of variables operators for exact embedding theorems, Integral Equations Operator Theory 19 (1994), 1--24.

\bibitem{GGR95} V.~Gol'dshtein, L.~Gurov, A.~Romanov, Homeomorphisms that induce monomorphisms of Sobolev spaces, 
Israel J. Math., 91 (1995), 31--60.

\bibitem{GResh} V. M. Gol'dshtein, Yu. G. Reshetnyak, Quasiconformal mappings and Sobolev spaces, Dordrecht, Boston, London: Kluwer Academic Publishers, (1990).

\bibitem{GS82} V.~M.~Gol'dshtein, V.~N.~Sitnikov, Continuation of functions of the class $W^1_p$ across H\"older boundaries, Imbedding theorems and their applica-ions, Trudy Sem. S. L. Soboleva, 1 (1982), 31--43.

\bibitem{GU09} V.~Gol'dshtein, A.~Ukhlov, Weighted Sobolev spaces and embedding theorems, Trans. Amer. Math. Soc., 361, (2009), 3829--3850.

\bibitem{GU17} V.~Gol'dshtein, A.~Ukhlov, The spectral estimates for the Neumann-Laplace operator in space domains, 
Adv. in Math., 315 (2017), 166--193.

\bibitem{H93} P.~Hajlasz, Change of variable formula under the minimal assumptions, Colloq. Math., 64 (1993), 93--101.

\bibitem{KM02} S.~Kallunki, O.~Martio, $ACL$ homeomorphisms and linear dilatation, Proc. Amer. Math. Soc., 130(2002), 1073--1078.

\bibitem{KR05} P.~Koskela, S.~Rogovin, Linear dilatation and absolute continuity, Ann. Acad. Sci. Fenn. Math., 30(2005), 385--392.

\bibitem{M} V.~Maz'ya, Sobolev spaces: with applications to elliptic partial differential equations, Springer: Berlin/Heidelberg, (2010).

\bibitem{MH72} V.~Maz'ya, V.~Havin, Nonlinear potential theory, Russian Math. Surveys, 27 (1972), 71--148.

\bibitem{MS86} V.~Maz'ya, T.~O.~Shaposhnikova, Multipliers in Spaces of Differentiable Functions, 
Leningrad Univ. Press., 1986.

\bibitem{RR55} T.~Rado, P.~V.~Reichelderfer, Continuous Transformations in Analysis. Springer-Verlag, Berlin (1955).

\bibitem{Va} Ju.~V\"ais\"al\"a, Lectires on $n$-dimensional quasiconformal mappings, Springer Verlag, (1955).

\bibitem{U93} A.~D.~Ukhlov, On mappings, which induce embeddings of Sobolev spaces, Siberian Math. J., 34 (1993), 185--192.

\bibitem{U94} A.~D.~Ukhlov, Sobolev spaces and mappings of Carnot groups associated with them, PhD Thesis, Novosibirks State University, Novosibirsk,  (1994).

\bibitem{V88} S.~K.~Vodop'yanov, Taylor Formula and Function Spaces, Novosibirsk Univ. Press., 1988.

\bibitem{VG75} S.~K.~Vodop'yanov, V.~M.~Gol'dstein, Lattice isomorphisms of the spaces $W^1_n$ and quasiconformal mappings,
Siberian Math. J., 16 (1975), 224--246.

\bibitem{VGR} S.~K.~Vodop'yanov, V.~M.~Gol'dshtein, Yu.~G.~Reshetnyak, On geometric properties of functions with generalized first derivatives, Uspekhi Mat. Nauk {34} (1979), 17--65. 

\bibitem{VU98} S.~K.~Vodop'yanov, A.~D.~Ukhlov, Sobolev spaces and $(P,Q)$-quasiconformal mappings of Carnot groups.
Siberian Math. J., 39 (1998), 665--682.

\bibitem{VU04} S.~K.~Vodop'yanov, A.~D.~Ukhlov, Set Functions and Their Applications in the Theory of Lebesgue and Sobolev Spaces. I, Siberian Adv. Math., 14 (2004), 78--125.

\bibitem{Vu} M.~Vuorinen, Conformal Geometry and Quasiregular Mappings, Lecture Notes in Math., vol.1319, Springer-Verlag, Berlin, Heidelberg, New York, 2006.

\end{thebibliography}
\end{document}